\documentclass{article}
\usepackage{amsmath}
\usepackage{amsfonts}
\usepackage{amsthm}
\usepackage{amssymb}
\usepackage{mathrsfs}
\usepackage{hyperref}
\usepackage[all]{xy}

\theoremstyle{definition}
\newtheorem{definition}{Definition}

\theoremstyle{remark}
\newtheorem{remark}{Remark}

\theoremstyle{plain}
\newtheorem{theorem}{Theorem}
\newtheorem{corollary}{Corollary}
\newtheorem{lemma}{Lemma}
\newtheorem{prop}{Proposition}

\DeclareMathOperator{\im}{Im}
\DeclareMathOperator{\re}{Re}

\DeclareMathOperator{\defeq}{\stackrel{\text{def\textsuperscript{\underline{n}}}}{=}}

\DeclareMathOperator{\loc}{loc}
\DeclareMathOperator{\spans}{span}

\begin{document}

\title{Weighted composition operators on spaces of analytic functions on the complex half-plane}
\author{Andrzej S. Kucik}
\date{}
\maketitle

\begin{flushright}
School of Mathematics \\
University of Leeds \\
Leeds LS2 9JT \\
United Kingdom \\
e-mail: \ttfamily{mmask@leeds.ac.uk}

\end{flushright}

\begin{abstract}
In this paper we will show how the boundedness condition for the \emph{weighted composition operators} on a class of spaces of analytic functions on the open right complex half-plane called \emph{Zen spaces} (which include the \emph{Hardy spaces} and \emph{weighted Bergman spaces}) can be stated in terms of \emph{Carleson measures} and \emph{Bergman kernels}. In Hilbertian setting we will also show how the norms of \emph{causal} weighted composition operators on these spaces are related to each other and use it to show that an \emph{(unweighted) composition operator} $C_\varphi$ is bounded on a Zen space if and only if $\varphi$ has a \emph{finite angular derivative at infinity}. Finally, we will show that there is no compact composition operator on Zen spaces. 
\end{abstract}

\section{Introduction}

Let $\mathcal{L}$ be a linear space of complex-valued functions defined on a domain $\Omega \subseteq \mathbb{C}$, let $\varphi$ be a self map of $\Omega$, and let $h : \Omega \rightarrow \mathbb{C}$. The \emph{weighted composition operator}, $W_{h, \varphi}$, on $\mathcal{L}$ (corresponding to symbols $\varphi$ and $h$) is defined to be the linear map
\begin{displaymath}
	W_{h, \varphi} f := h \cdot (f \circ \varphi), \qquad \qquad f \in \mathcal{L}.
\end{displaymath}
If $h \equiv 1$, then we write $W_{1, \, \varphi} := C_\varphi$ to denote the (unweighted) composition operator on $\mathcal{L}$.
The composition operators are usually studied in the context of spaces of analytic functions (we then require that both $\varphi$ and $h$ are also analytic functions). The book \cite{cowen1995} is an excellent source in that instance. In fact, the particular case of weighted and unweighted composition operators on spaces of analytic functions defined on $\Omega = \mathbb{D} \defeq \{z \in \mathbb{C} \; : \: |z|<1 \}$, the open unit disk of the complex plane, have been studied very extensively. It follows from the Littlewood subordination principle (see \cite{littlewood1925}) that all unweighted composition operators are bounded on all Hardy spaces and weighted Bergman spaces (see \cite{cowen1995}).

In this article we would like to consider the domain
\begin{displaymath}
	\mathbb{C}_+ :=\left\{ z \in \mathbb{C} \; : \; \re(z) >0 \right\}, 
\end{displaymath}
the \emph{open right complex half-plane}. Unlike the case of Hardy or weighted Bergman spaces on $\mathbb{D}$, not every analytic composition operator is bounded on Hardy or weighted Bergman spaces on $\mathbb{C}_+$. This was proved in \cite{elliott2012} by S. Elliott and M. T. Jury (for the Hardy spaces) and by S. Elliott and A. Wynn in \cite{elliott2011} (for weighted Bergman spaces). In particular, they have found the expression for the norm of composition operators for these spaces and shown that the composition operators must never be compact in this context. Some further results concerning so-called \emph{Zen spaces} (a generalisation of weighted Bergman spaces) and weighted composition operators have been obtained in \cite{chalendar2014} and we aim to extend it in this article.

In Section \ref{sec:pre} of this paper we will introduce the Zen spaces. In Hilbertian setting, we will state their connection with weighted $L^2$ spaces on $(0, \, \infty)$ and give the formula for their reproducing kernels. In Section \ref{sec:carleson} we will define the notion of a \emph{Carleson measure} and show how it is related to boundedness of weighted composition operators. We will also use this relation to give a boundedness criterion in terms of \emph{Bergman kernels}. In Section \ref{sec:causality} we will define the concept of causality and show how boundedness of weighted and unweighted composition operators on one space can imply the boundedness on the whole class of spaces. In particular, we will use it to show that an unweighted composition operator $C_\varphi$ is bounded on a Zen space $A^2_\nu$ if and only if $\varphi$ has a \emph{finite angular derivative at infinity}. We will also estimate the norms of bounded composition operators in this case. Finally, in Section \ref{sec:compactness} we will show that there is no compact composition operator on a Zen space $A^2_\nu$.

Throughout this paper $h$ is always defined to be a holomorphic map on $\mathbb{C}_+$ and $\varphi$ a holomorphic self-map on $\mathbb{C}_+$.

\section{Preliminaries}
\label{sec:pre}

The classical Hardy and weighted Bergman spaces defined on the open unit disk of the complex plane and the open right complex half-plane may be viewed as discrete and continuous counterparts (this is discussed for example in \cite{kucik2015} and \cite{kucik2016}). The continuous case is sometimes more appropriate when we consider applications of these spaces (see \cite{jacob2013}, \cite{jacob2014} and \cite{kucik2016*}).

Let $\tilde{\nu}$ be a positive regular Borel measure on $[0, \, \infty)$ satisfying the so-called \emph{$\Delta_2$-condition}:
\begin{equation}
	R:=\sup_{r>0} \frac{\tilde{\nu}[0, \, 2r)}{\tilde{\nu}[0, \, r)}  < \infty,
	\tag{$\Delta_2$}
	\label{eq:delta2}
\end{equation}
and let $\lambda$ denote the Lebesgue measure on $i\mathbb{R}$. We define $\nu := \tilde{\nu} \otimes \lambda$ to be a positive regular Borel measure on the closed right complex half-plane $\overline{\mathbb{C}_+}~:=~[0, \, \infty) \times i\mathbb{R}$. For this measure and $1 \leq p < \infty$, a \emph{Zen space} (see \cite{jacob2013}) is defined to be:
\begin{displaymath}
	A^p_\nu := \left\{F : \mathbb{C}_+ \longrightarrow \mathbb{C} \, \text{analytic} \; : \; \left\|F\right\|^p_{A^p_\nu} := \sup_{\varepsilon >0} \int_{\mathbb{C}_+} \left|F(z+\varepsilon)\right|^p \, d\nu < \infty \right\}.
\end{displaymath}
The Zen spaces may be viewed as an extension of the class of weighted Bergman spaces or the weighted Hardy spaces $H^2(\beta)$ (these are discussed in \cite{cowen1995}) if $p=2$ (see \cite{chalendar2014}).

In \cite{jacob2013} (Proposition 2.3) it is shown that the Laplace transform ($\mathfrak{L}$) defines an isometric map
\begin{displaymath}
	\mathfrak{L} : L^2_w(0, \, \infty) \longrightarrow A^2_\nu,
\end{displaymath}
where
\begin{displaymath}
	L^2_w(0, \infty) := \left\{f : (0, \, \infty) \longrightarrow \mathbb{C} \; : \; \|f\|_{L^2_w(0, \, \infty)} := \sqrt{\int_0^\infty |f(t)|^2 w(t) \, dt} < \infty \right\},
\end{displaymath}
and
\begin{equation}
	w(t) := 2\pi \int_0^\infty e^{-2rt} \, d\tilde{\nu}(r) \qquad \qquad (\forall t>0).
	\label{eq:w}
\end{equation}

If the Laplace transform is surjective, then $A^2_\nu$ is a reproducing kernel Hilbert space, and its kernels are given by
\begin{equation}
k^{A^2_\nu}_z (\zeta) : = \int_0^\infty \frac{e^{-t(\zeta+\overline{z})}}{w(t)} \, dt \qquad \qquad ((z, \, \zeta) \in \mathbb{C}_+^2)
\label{eq:kernel}
\end{equation}
(see \cite{kucik2015}).

\section{Carleson measures and boundedness}
\label{sec:carleson}

In \cite{contreras2001} M. D. Contreras and A. G. Hern\'{a}ndez-D\'{i}az gave a necessary and sufficient condition for a weighted composition operator $W_{h, \varphi}$ to be bounded on $H^p$, the Hardy spaces on the unit disk, in terms of so-called \emph{Carleson measures}. In Lemma \ref{lemma1} and Theorem \ref{thm:first} we shall state these results for the Zen spaces, omitting the proofs, as they are conceptually the same as proofs of Lemma 2.1 and Theorem 2.2 from \cite{contreras2001}.

\begin{definition}
Let $\mu$ be a positive Borel measure on $\mathbb{C}_+$. If there exists $C(\mu)~>~0$, such that
\begin{displaymath}
	\int_{\mathbb{C}_+} |F|^p \, d\mu \leq C(\mu) \|F\|^p_{A^p_\nu} \qquad \qquad (\forall F \in A^p_\nu),
\end{displaymath}
then we say that $\mu$ is a \emph{Carleson measure} for $A^p_\nu$.
\end{definition}

The notion of a Carleson measure has been introduced by Lennart Carleson in \cite{carleson1962} to solve the corona problem, and since then has found many other applications (for example, in the context of spaces of analytic functions on the half-plane, it is used to describe the admissibility criterion for control and observation operators, see \cite{jacob2013}, \cite{jacob2014}, \cite{kucik2016}, \cite{kucik2016*}).

\begin{lemma}
\label{lemma1}
Let  $\nu$ be a positive Borel measure on $\mathbb{C}_+$ and let $g$ be a non-negative measurable function on $\mathbb{C}_+$. Then
\begin{equation}
	\int_{\mathbb{C}_+} g \, d\mu_{\nu,h,\varphi,p} = \int_{\mathbb{C}_+} |h|^p (g \circ \varphi) \, d\nu,
\label{eq:compmu}
\end{equation}
where $\mu_{\nu,h,\varphi, p}$ is given by
\begin{displaymath}
	\mu_{\nu,h,\varphi, p}(E) = \int_{\varphi^{-1}(E)} |h|^p \, d\nu,
\end{displaymath}
for each Borel set $E \subseteq {\mathbb{C}_+}$.
\end{lemma}

\begin{theorem}
\label{thm:first}
The weighted composition operator $W_{h, \varphi}$ is bounded on a Zen space $A^p_\nu$ if and only if $\mu_{\nu,h,\varphi,p}$ is a Carleson measure for $A^p_\nu$.
\end{theorem}

The Carleson measures for Zen spaces are characterised in \cite{jacob2013}. Verifying the boundedness of a weighted composition operator using Carleson measures only can be very difficult in practice however. We shall therefore present a more useful approach.

\begin{definition}
Let $1\leq p< \infty$. A weighted Bergman space $\mathcal{B}^p_\alpha(\mathbb{C}_+)$ is defined to be
\begin{displaymath}
	\mathcal{B}^p_\alpha(\mathbb{C}_+) := \left\{F : \mathbb{C}_+ \longrightarrow \mathbb{C} \, \text{analytic} \; : \;  \|F\|^p_{\mathcal{B}^p_\alpha} := \int_{\mathbb{C}_+} (\re(z))^\alpha |F(z)|^p  \, \frac{dz}{\pi} < \infty \right\},
\end{displaymath}
for $\alpha > -1$, and
\small
\begin{displaymath}
	\mathcal{B}^p_{-1}(\mathbb{C}_+) := \left\{F : \mathbb{C}_+ \longrightarrow \mathbb{C} \, \text{analytic} \; : \;  \|F\|^p_{\mathcal{B}^p_{-1}} := \sup_{x>0}\int_{-\infty}^\infty |F(x+iy)|^2  \, \frac{dy}{2\pi} < \infty \right\},
\end{displaymath}
\normalsize
that is $\mathcal{B}^p_{-1}(\mathbb{C}_+) \defeq H^p(\mathbb{C}_+)$, the Hardy space on the open right complex half-plane. It is easy to see that if $\alpha>-1$ and $d\tilde{\nu}(r) = r^\alpha \, dr / \pi$, or $\alpha=-1$ and $\tilde{\nu} = \delta_0 /2\pi$ ($\delta_0$ is the Dirac delta measure in 0), then $\mathcal{B}^p_\alpha(\mathbb{C}_+) = A^p_\nu$. If $p=2$, then the weighted Bergman spaces are reproducing kernel Hilbert spaces, with kernels given by
\begin{displaymath}
	k^{\mathcal{B}^2_\alpha(\mathbb{C}_+)}_z(\zeta) \stackrel{\eqref{eq:kernel}}{=} \frac{2^\alpha(1+\alpha)}{(\overline{z}+\zeta)^{2+\alpha}} \qquad \qquad (z, \, \zeta \in \mathbb{C}_+, \; \alpha>-1)
\end{displaymath}
and
\begin{displaymath}
	k^{\mathcal{B}^2_{-1}(\mathbb{C}_+)}_z(\zeta) \defeq k^{H^2(\mathbb{C}_+)}_z(\zeta) \stackrel{\eqref{eq:kernel}}{=} \frac{1}{\overline{z}+\zeta} \qquad \qquad (z, \, \zeta \in \mathbb{C}_+).
\end{displaymath}
\end{definition}

\begin{theorem}
The weighted composition operator $W_{h, \varphi}$ is bounded on a Zen space $A^p_\nu$ if and only if there exists $\alpha \geq -1$ such that
\begin{equation}
	\Lambda(\alpha) := \sup_{z \in \mathbb{C}_+} \frac{\left\|h \cdot \left(k^{\mathcal{B}^2_\alpha(\mathbb{C}_+)}_z \circ \varphi\right) \right\|_{A^p_\nu}}{\left\|k^{\mathcal{B}^2_\alpha(\mathbb{C}_+)}_z \right\|_{A^p_\nu}} < \infty.
\label{eq:boundedkernel}
\end{equation}
\end{theorem}

\begin{proof}
By the proof of Theorem 2.1 from \cite{jacob2013} we get that there exists $N \in \mathbb{N}$ such that
\begin{equation}
	\left\|\frac{1}{(\overline{z}+\zeta)^N}\right\|^p_{A^p_\nu} \leq \nu\left(\overline{Q(z)}\right)\left(\frac{2}{\re(z)}\right)^{Np} \sum_{k=0}^\infty (2^{1-Np}R)^k,
	\label{eq:N}
\end{equation}
where $R \geq 1$ is the same quantity as in \eqref{eq:delta2}. Theorem 2.1 from \cite{jacob2013} defines $N$ to be a large natural number, but it is easy to see that this requirement is unnecessarily strict. In fact, it suffices to choose any real $N > (1+\log_2(R))/p$ to make the above sum converge. Therefore, if $\alpha > \max\{-1, \, (1+\log_2(R))/p -2\}$, we have that $k^{B^2_{\alpha}(\mathbb{C}_+)}_z$ is in $A^p_\nu$, for all $z \in \mathbb{C}_+$. So if the weighted composition operator is bounded on $A^p_\nu$, then $\Lambda(\alpha)$ must be finite. Conversely, if $\Lambda(\alpha)$ is finite for some $\alpha \geq-1$, then
\begin{displaymath}
	\int_{\mathbb{C}_+} \left|k^{B^2_{\alpha}(\mathbb{C}_+)}_z\right|^p \, d\mu_{\nu, h, \varphi, p} \stackrel{\eqref{eq:compmu}}{=} \int_{\mathbb{C}_+} \left|h \cdot \left(k^{B^2_{\alpha}(\mathbb{C}_+)}_z \circ \varphi \right)\right|^p \, d\nu \stackrel{\eqref{eq:boundedkernel}}{\leq} \Lambda_p^p \int_{\mathbb{C}_+} \left|k^{B^2_{\alpha}(\mathbb{C}_+)}_z\right|^p \, d\nu,
\end{displaymath}
so, again by Theorem 2.1 from \cite{jacob2013}, the canonical embedding \\ 
$A^p_\nu \hookrightarrow L^p(\mathbb{C}_+, \, \mu_{\nu,h,\varphi,p})$ is bounded, and hence, by Theorem \ref{thm:first}, $W_{h, \varphi}$ is bounded on $A^p_\nu$. 
\end{proof}

\begin{remark}
\label{remark:1}
If $p>1$ and $A^p_\nu = B^p_{\alpha}(\mathbb{C}_+)$ for some $\alpha \geq -1$, then we can use this $\alpha$, setting $N=\alpha+2$, to make the sum in \eqref{eq:N} converge. In this case we have that $W_{h, \varphi}$ is bounded on $B^p_{\alpha}(\mathbb{C}_+)$ if and only if
\begin{equation}
	\sup_{z \in \mathbb{C}_+} \frac{\left\|h \cdot \left(k^{\mathcal{B}^2_\alpha(\mathbb{C}_+)}_z \circ \varphi\right) \right\|_{\mathcal{B}^p_\alpha(\mathbb{C}_+)}}{\left\|k^{\mathcal{B}^2_\alpha(\mathbb{C}_+)}_z \right\|_{\mathcal{B}^p_\alpha(\mathbb{C}_+)}} < \infty.
	\label{eq:nonhilbertian}
\end{equation}
In particular, if $p=2$, \eqref{eq:nonhilbertian} is equivalent to
\begin{displaymath}
	\sup_{z \in \mathbb{C}_+} |h(z)| \left(\frac{\re(z)}{\re(\varphi(z))}\right)^{\alpha+2} < \infty,
\end{displaymath}
since $W^*_{h,\varphi}k^{\mathcal{B}^2_\alpha(\mathbb{C}_+)}_z = \overline{h(z)} k^{\mathcal{B}^2_\alpha(\mathbb{C}_+)}_{\varphi(z)}$.
\end{remark}

\begin{definition}
A sequence of points $z_n=x_n+iy_n \in \mathbb{C}_+$ is said to approach $\infty$ non-tangentially if $\lim_{n \rightarrow \infty}x_n = \infty$ and $\sup_{n \in \mathbb{N}} |y_n|/x_n < \infty$. We also say that $\varphi$ fixes infinity non-tangentially if $\varphi(z_n) \rightarrow \infty$ whenever $z_n \rightarrow \infty$ non-tangentially, and write $\varphi(\infty)=\infty$. If it is the case and also the non-tangential limit
\begin{equation}
	\lim_{z \rightarrow \infty} \frac{z}{\varphi(z)}
\label{eq:phi'}
\end{equation}
exists and is finite, then we say that $\varphi$ has a finite angular derivative at infinity and denote the above limit by $\varphi'(\infty)$.
\end{definition}

\begin{prop}[Julia-Carath\'{e}odory Theorem in $\mathbb{C}_+$ - Proposition 2.2 in \cite{elliott2012}]
\label{prop:elliott}
Let $\varphi$ be an analytic self-map on $\mathbb{C}_+$. The following are equivalent:
\begin{enumerate}
	\item $\varphi(\infty)$ and $\varphi'(\infty)$ exist;
	\item $ \sup_{z \in \mathbb{C}_+} = \frac{\re(z)}{\re(\varphi(z))} < \infty$;
	\item $ \limsup_{z \longrightarrow \infty} = \frac{\re(z)}{\re(\varphi(z))} < \infty$.
\end{enumerate}
Moreover, the quantities in 2. and 3. are both equal to $\varphi'(\infty)$.
\end{prop}

\begin{corollary}
A composition operator $C_\varphi$ is bounded on $B^2_\alpha(\mathbb{C}_+)$ if and only if $\varphi$ has a finite angular derivative at infinity.
\end{corollary}

\begin{proof}
It follows from the above proposition and Remark \ref{remark:1}.
\end{proof}

This was proved (although using different approach) in \cite{elliott2012} (for the Hardy spaces $H^p(\mathbb{C}_+)$) and in \cite{elliott2011} (for weighted Bergman spaces $B^2_\alpha(\mathbb{C}_+$)). In the next section we shall prove that it also holds for all Zen spaces.

\section{Causality}
\label{sec:causality}

\begin{definition}
Let $w$ be a positive measurable function on $(0, \, \infty)$. We say that $A:L^2_{w}(0, \infty) \rightarrow L^2_{w}(0, \, \infty)$ is a \emph{causal operator} (or a \emph{lower-triangular operator}), if for each $T>0$ the closed subspace $L^2_{w}(T, \, \infty)$ is invariant for $A$. If there exists $\alpha>0$ such $AL^2_{w}(T, \, \infty) \subseteq L^2_w(T+\alpha, \, \infty)$, for all $T>0$, then we say that $A$ is \emph{strictly causal}.
\end{definition}

The following lemma was proved in \cite{chalendar2014} (Theorem 3.2) for unweighted $L^2$ spaces and we shall modify the proof to extend the result to weighted $L^2$ spaces.

\begin{lemma}
Let $w \in L^1_{\loc}(0, \, \infty)$ be positive and non-increasing. Suppose that $A : L^2_{w}(0, \infty) \rightarrow L^2_{w}(0, \infty)$ is a causal operator and $D$ is the operator of multiplication by a strictly positive, monotonically increasing function $d$. Then
\begin{equation}
	\left\|D^{-1}AD\right\|_{L^2_{w}(0, \infty)} \leq \left\|A\right\|_{L^2_{w}(0, \infty)}.
	\label{eq:dad}
\end{equation}
\label{lemma:2}
\end{lemma}

\begin{proof}
Suppose first that $A$ is strictly causal for some $\alpha>0$. For $\re(z)\geq 0$ define $\Omega(z) = D^{-z}AD^z$, where $D^z$ is the operator of multiplication by the complex function $d^z$. For each $N \in \mathbb{N}$ such that $N \geq \log_2\alpha$ let
\begin{displaymath}
	X_N := \spans \left\{e_k : = \chi_{(k/2^N, \, (k+1)/2^N)} \; : \; 1 \leq k \leq 2^{2N} \right\}.
\end{displaymath}
Let $P_N : L^2_w(0, \, \infty) \longrightarrow X_N$ denote the orthogonal projection and define \\ 
$\Omega_N(z)~=~\Omega(z)P_N$, which maps each $e_k$ to $d^{-z}Ad^ze_k$ and
\begin{displaymath}
	\left\|d^{-z}A d^z e_k\right\| \leq \|A\| \|e_k\|,
\end{displaymath}
since $d$ is increasing and $A$ is strictly causal. $X_N$ is finite dimensional, so $\Omega_N(z)$ is bounded independently of $z$, because $\|\Omega_N(z)\| \leq \|\Omega(z)|_{X_N}\|$. By the maximum principle we also have that
\begin{displaymath}
	\|\Omega_N(1)\| \leq \sup_{\re(z) \geq 0} \|\Omega_N(z)\| \leq \sup_{\re(z) = 0} \|\Omega(z)\| = \|A\|,
\end{displaymath}
and the result holds on $L^2_w(0, \, \infty)$, since $\bigcup_{N \geq \log_2 \alpha}X_N$ is a dense set therein.

If $A$ is not strictly causal, then let $S_\alpha$ denote the right shift by $\alpha$. In this case the operator $S_\alpha A$ is strictly causal and, by the above, we have
\begin{displaymath}
	\|D^{-1}S_\alpha A D\| \leq \|S_\alpha A\| = \|A\|.
\end{displaymath}
Let $d_\alpha(t) = d(\alpha+t)$. Then for each $f \in L^2_w(0, \, \infty)$ we have $\|D^{-1}S_\alpha f\|=\|d_\alpha^{-1}f\|$, and $|d_\alpha^{-1} f \sqrt{w}|$ increases to $|d^{-1}f\sqrt{w}|$ almost everywhere as $\alpha \rightarrow 0$, because the monotonically decreasing function $d^{-1}$ is continuous almost everywhere, and hence the result follows from Lebesgue's monotone convergence theorem.
\end{proof}

We will say that an operator $B : A^2_\nu \longrightarrow A^2_\nu$ is causal if the the corresponding isometric operator $\mathfrak{L}^{-1}B\mathfrak{L} : L^2_w(0, \, \infty) \longrightarrow L^2_w(0, \, \infty)$ is causal on $L^2_w(0, \, \infty)$

\begin{displaymath}
	\xymatrix{L^2_w(0, \, \infty) \ar[r]^{\mathfrak{L}^{-1} B \mathfrak{L}} \ar[d]|{\mathfrak{L}} & L^2_w(0, \, \infty) \\
						A^2_\nu \ar[r]_{B} & A^2_\nu \ar[u]|{\mathfrak{L}^{-1}}}
\end{displaymath}

From now on we will assume that the Laplace transform $\mathfrak{L}~:~L^2_w(0, \, \infty)~\rightarrow~A^2_\nu$ (where $w$ is as given in \eqref{eq:w}) is surjective. The surjectivity of $\mathfrak{L}$ in this context is discussed in \cite{zen2009}.

\begin{theorem}
Suppose that the weighted composition operator $W_{h,\varphi}$ is bounded and causal on $A^2_\nu$. Then there exists $\alpha' \geq 0$ such that for each $\alpha \geq \alpha'$  $W_{h,\varphi}$ is bounded on the weighted Bergman space $\mathcal{B}^2_{\alpha}(\mathbb{C}_+)$, and
\begin{equation}
	\left\|W_{h,\varphi}\right\|_{\mathcal{B}^2_\alpha(\mathbb{C}_+)} \leq \left\|W_{h,\varphi}\right\|_{A^2_\nu}.
\label{eq:weightedbergmanzen}
\end{equation}
\label{thm:awfultheorem}
\end{theorem}

\begin{proof}
Let $L^2_{w}(0, \infty), \; L^2_{v_\alpha}(0, \infty)$ be the spaces corresponding to $A^2_\nu$ and $\mathcal{B}^2_\alpha(\mathbb{C}_+)$ respectively (that is $v_\alpha(t) = 2^{-\alpha}\Gamma(\alpha+1) t^{-\alpha-1}$). We want to show that \\
$\sqrt{w(t)/v_\alpha(t)}$ is an increasing function, that is, we must have
\begin{align*}
	w'(t)v_\alpha(t) &\geq w(t)v'_\alpha(t) \\
-2\pi \int_0^\infty 2re^{-2rt} \, d\tilde{\nu}(r) \cdot \frac{\Gamma(\alpha+1)}{2^\alpha t^{\alpha+1}} &\geq -2\pi \int_0^\infty e^{-2rt} \, d\tilde{\nu}(r) \cdot \frac{\Gamma(\alpha+2)}{2^\alpha t^{\alpha+2}} \\
	\int_0^\infty e^{-2rt}\left(r-\frac{\alpha+1}{2t}\right) \, d\tilde{\nu}(r) &\leq 0.
\end{align*}
Consider the graph:

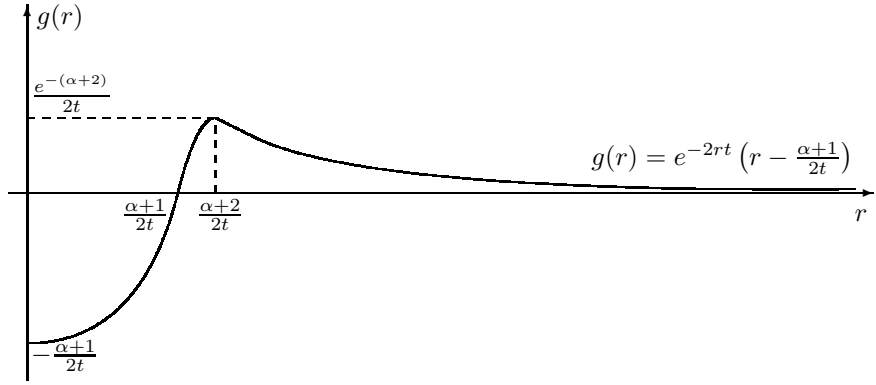
\setlength{\unitlength}{5mm}
\begin{picture}(25,12)(-0.5,-6)
\put(-0.5,0){\vector(1,0){23}}
\put(0,-5){\vector(0,1){10}}
\put(22,-0.75){$r$}
\put(0.25,4.5){$g(r)$}
\put(15,0.75){$g(r)=e^{-2rt}\left(r-\frac{\alpha+1}{2t}\right)$}
\put(0.1,-4.5){$-\frac{\alpha+1}{2t}$}
\put(2.5,-0.75){$\frac{\alpha+1}{2t}$}
\put(4.5,-0.75){$\frac{\alpha+2}{2t}$}
\put(0.1,2.4){$\frac{e^{-(\alpha+2)}}{2t}$}
\multiput(0,2)(0.4,0){13}
{\line(1,0){0.2}}
\multiput(5,0)(0,0.4){5}
{\line(0,1){0.2}}
\linethickness{0.2mm}
\qbezier(0,-4)(3,-4)(4,0)
\qbezier(4,0)(4.5,2)(5,2)
\qbezier(5,2)(5,2)(6,1.5)
\qbezier(6,1.5)(9,0)(22,0.1)
\end{picture}

We clearly need to have
\begin{displaymath}
	-\int_0^{\frac{\alpha+1}{2t}} e^{-2rt}\left(r-\frac{\alpha+1}{2t}\right) \, d\tilde{\nu}(r) \geq \int_{\frac{\alpha+1}{2t}}^\infty e^{-2rt}\left(r-\frac{\alpha+1}{2t}\right) \, d\tilde{\nu}(r).
\end{displaymath}
Observe that for $\alpha \geq 0$ we have 
\begin{equation}
	\frac{\alpha+2}{4t} \leq \frac{\alpha+1}{2t}.
\label{eq:alphaest}
\end{equation}
Let $R$ be defined as in \eqref{eq:delta2}. Then we have
\begin{displaymath}
	-\int_0^{\frac{\alpha+1}{2t}} e^{-2rt}\left(r-\frac{\alpha+1}{2t}\right) \, d\tilde{\nu}(r) \geq \tilde{\nu}\left[0, \frac{\alpha+2}{4t}\right) \alpha \frac{e^{-\frac{\alpha+2}{2}}}{2t} \stackrel{\eqref{eq:delta2}}{\geq} \frac{\tilde{\nu}\left[0, \frac{\alpha+2}{2t}\right)}{2Rt} \alpha e^{-\frac{\alpha+2}{2}},
\end{displaymath}
\begin{align*}
	\int_{\frac{\alpha+1}{2t}}^{\frac{\alpha+2}{2t}} e^{-2rt}\left(r-\frac{\alpha+1}{2t}\right) \, d\tilde{\nu}(r) &\leq \tilde{\nu}\left[\frac{\alpha+1}{2t},\frac{\alpha+2}{2t}\right) \frac{e^{-(\alpha+2)}}{2t} \\
	&= \left(\tilde{\nu}\left[0,\frac{\alpha+2}{2t} \right) - \tilde{\nu}\left[0,\frac{\alpha+1}{2t} \right) \right) \frac{e^{-(\alpha+2)}}{2t} \\
	&\leq  \frac{\tilde{\nu}\left[0, \frac{\alpha+2}{2t}\right)}{2Rt} (R-1) e^{-(\alpha+2)},
	\end{align*}
because
\begin{displaymath}
	- \tilde{\nu}\left[0, \frac{\alpha+1}{2t}\right) \stackrel{\eqref{eq:alphaest}}{\leq} - \tilde{\nu}\left[0, \frac{\alpha+2}{4t}\right) \stackrel{\eqref{eq:delta2}}{\leq} -\frac{\tilde{\nu}\left[0, \frac{\alpha+2}{2t}\right)}{R},
\end{displaymath}
and
\begin{align*}
\int_{\frac{\alpha+2}{2t}}^\infty e^{-2rt}\left(r-\frac{\alpha+1}{2t}\right) \, d\tilde{\nu}(r) &\leq \sum_{n=0}^\infty \tilde{\nu}\left[2^n \frac{\alpha+2}{2t}, 2^{n+1}\frac{\alpha+2}{2t} \right) e^{-2^n(\alpha+2)}\frac{2^n(\alpha+2)-\alpha-1}{2t} \\
&\stackrel{\eqref{eq:delta2}}{\leq} \frac{\tilde{\nu}[0, \frac{\alpha+2}{2t})}{2t} (R-1) (\alpha+2) e^{-(\alpha+2)} \sum_{n=0}^\infty \left(2Re^{-(\alpha+2)}\right)^n \\
&= \frac{\tilde{\nu}[0, \frac{\alpha+2}{2t})}{2t} (R-1) e^{-(\alpha+2)}  \frac{\alpha+2}{1-2Re^{-(\alpha+2)}}.
\end{align*}
Collecting these inequalities we get
\begin{displaymath}
	e^{-\frac{\alpha+2}{2}} \frac{R-1}{\alpha} \left(1+\frac{R(\alpha+2)}{1-2Re^{-(\alpha+2)}}\right) \leq 1,
\end{displaymath}
which is true for all $\alpha \geq \alpha'$, for some $\alpha'$ sufficiently large. Now, let $A$ be an operator on $L^2_{v_\alpha}(0, \infty)$ induced by $W_{h,\varphi}$ acting on $\mathcal{B}^2_\alpha(\mathbb{C}_+)$ and let $D$ be the isometric operator from $L^2_w(0, \infty)$ to $L^2_{v_\alpha}(0, \infty)$ of multiplication by $\sqrt{w(t)/v_\alpha(t)}$. Consider the following commutative diagram:
\begin{displaymath}
	\xymatrix{L^2_w(0, \infty) \ar[r]^D \ar[d]|{D^{-1}AD} & L^2_{v_\alpha}(0, \infty) \ar[d]|A \\
						L^2_w(0, \infty) & L^2_{v_\alpha}(0, \infty) \ar[l]_{D^{-1}}}
\end{displaymath}
By Lemma \ref{lemma:2} we therefore have
\begin{displaymath}
	\left\|W_{h,\varphi}\right\|_{\mathcal{B}^2_\alpha(\mathbb{C}_+)} = \|A\|_{L^2_{v_\alpha}(0, \, \infty) } = \left\|D^{-1}AD\right\|_{L^2_w(0, \infty)} \stackrel{\eqref{eq:dad}}{\leq} \left\|A\right\|_{L^2_w(0, \infty)} = \left\|W_{h,\varphi}\right\|_{A^2_\nu},
\end{displaymath}
as required.
\end{proof}

In the remaining part of this section we will use the \emph{Nevanlinna representation} of a holomorphic function $\varphi :\mathbb{C}_+ \longrightarrow \mathbb{C}_+$:
\begin{equation}
	\varphi(z) = az+ib+ \int_\mathbb{R} \left(\frac{1}{it+z}+\frac{it}{1+t^2}\right) d\mu(t) = az+ib + \int_\mathbb{R} \frac{1+itz}{it+z} \frac{d\mu(t)}{1+t^2},
\label{eq:nevanlinna}
\end{equation}
where $a \geq 0, \, b \in \mathbb{R}$ and $\mu$ is a non-negative Borel measure measure on $\mathbb{R}$ satisfying the following growth condition:
\begin{displaymath}
	\int_\mathbb{R} \frac{d\mu(t)}{1+t^2} < \infty
\end{displaymath}
(see \cite{krein1977}). Clearly
\begin{displaymath}
	a = \lim_{\re(z) \rightarrow \infty} \frac{\varphi(\re(z))}{\re(z)}.
\end{displaymath}

\begin{theorem}[Theorem 3.1 in \cite{elliott2012} and Theorem 3.4 in \cite{elliott2011}]
The composition operator $C_\varphi$ is bounded on $\mathcal{B}^2_{\alpha}(\mathbb{C}_+), \, \alpha \geq -1$, if and only if $\varphi$ has finite angular derivative
$0 < \lambda < \infty$ at infinity, in which case $\|C_\varphi\| = \lambda^{(2+\alpha)/2}$.
\label{thm:elliott}
\end{theorem}

Let $a>0$. We define a holomorphic map $\psi_a : \mathbb{C}_+ \longrightarrow \mathbb{C}_+$ by $\psi_a(z) = az$, for all $z \in \mathbb{C}_+$.

\begin{prop}[Corollary 3.4 in \cite{chalendar2014}]
\label{prop:chalendar1}
Let $a>0$. If $W_{h,\varphi}$ is bounded on $H^2(\mathbb{C}_+)$, then it is also bounded on $A^2_\nu$ and
\begin{equation}
\|W_{h,\varphi}\|_{A^2_\nu} \leq \|C_{\psi_a}\|_{A^2_\nu} \|C_{\psi_{1/a}}\|_{H^2(\mathbb{C}_+)} \|W_{h, \varphi}\|_{H^2(\mathbb{C}_+)}.
\label{eq:chalendar1}
\end{equation}
\end{prop}

\begin{prop}[Proposition 3.5 in \cite{chalendar2014}]
\label{prop:chalendar2}
Let $a > 0$. Then
\begin{equation}
\|C_{\psi_a}\|_{A^2_\nu} = \sqrt{\sup_{t>0} \frac{w(at)}{aw(t)}}.
\label{eq:chalendar2}
\end{equation}
\end{prop}

\begin{corollary}
Let $a>0$. If $W_{h, \varphi}$ is bounded on some Zen space $A^2_\nu$, then there exists $\alpha'>0$ such that for all $\alpha \geq \alpha'$, we have
\begin{displaymath}
\left\|W_{h,\varphi}\right\|_{\mathcal{B}^2_\alpha(\mathbb{C}_+)} \leq a^{-\frac{\alpha+1}{2}} \left(\sup_{t>0} \frac{w(t/a)}{w(t)}\right)^{1/2} \left\|W_{h,\varphi}\right\|_{A^2_\nu} < \infty.
\end{displaymath}
\label{cor:nicecorollary}
\end{corollary}

\begin{proof}
The operator $C_{\psi_\frac{1}{a}} W_{h,\varphi}$ is causal for each $a>0$, so by Theorem \ref{thm:awfultheorem} we have
\begin{displaymath}
\left\|W_{h,\varphi}\right\|_{\mathcal{B}^2_\alpha(\mathbb{C}_+)} \stackrel{\eqref{eq:weightedbergmanzen}}{\leq} \left\|C_{\psi_a}\right\|_{\mathcal{B}^2_\alpha(\mathbb{C}_+)} \left\|C_{\psi_\frac{1}{a}}\right\|_{A^2_\nu} \left\|W_{h,\varphi}\right\|_{A^2_\nu},
\end{displaymath}
and the result follows from Theorem \ref{thm:elliott} and Proposition \ref{prop:chalendar2}.
\end{proof}

\begin{theorem}
The composition operator $C_\varphi$ is bounded on $A^2_\nu$ if and only if $\varphi$ has a finite angular derivative $0 < \lambda < \infty$ at infinity. If $C_\varphi$ is bounded, then
\begin{displaymath}
	\lambda \inf_{t>0} \frac{w(t)}{w(\lambda t)} \leq \left\|C_\varphi\right\|^2_{A^2_\nu} \leq \lambda \sup_{t>0}\frac{w(t/\lambda)}{w(t)},
\end{displaymath}
\label{thm:Zenboundedness}
\end{theorem}

\begin{proof}
Suppose, for contradiction, that $C_\varphi$ is bounded on $A^2_\nu$, but $\varphi$ does not have a finite angular derivative at infinity. By Proposition \ref{prop:elliott} we know that for each $n \geq 1 $ there must exist $z_n \in \mathbb{C}_+$ such that
\begin{equation}
	\frac{\re(z_n)}{\re(\phi(z_n))} > n.
	\label{eq:z_n}
\end{equation}
Now,
\begin{equation}
	\|C^*_\varphi\|^2 \geq \frac{\left\|k^{A^2_\nu}_{\varphi(z_n)}\right\|^2}{\left\|k^{A^2_\nu}_{z_n}\right\|^2} \stackrel{\eqref{eq:kernel}}{=} \frac{\int_0^\infty \frac{e^{-2t\re(\varphi(z_n))}}{w(t)} \, dt}{\int_0^\infty \frac{e^{-2t\re(z_n)}}{w(t)} \, dt} \stackrel{\eqref{eq:z_n}}{\geq} \frac{\int_0^\infty \frac{e^{-2t\re(z_n)/n}}{w(t)} \, dt}{\int_0^\infty \frac{e^{-2t\re(z_n)}}{w(t)} \, dt}.
	\label{eq:Cphin}
\end{equation}
Since $w$, by definition, is non-increasing, we have that $w(nt) \leq w(t)$, for all $n \geq 1$, and consequently
\begin{displaymath}
	\|C^*_\varphi\|^2 \stackrel{\eqref{eq:Cphin}} \geq  \frac{\int_0^\infty \frac{e^{-2t\re(z_n)}}{w(nt)} n\, dt}{\int_0^\infty \frac{e^{-2t\re(z_n)}}{w(t)} \, dt} \geq n \frac{\int_0^\infty \frac{e^{-2t\re(z_n)}}{w(t)} \, dt}{\int_0^\infty \frac{e^{-2t\re(z_n)}}{w(t)} \, dt} = n,
\end{displaymath}
for all $n \geq 1$, which is absurd, as it contradicts the boundedness of $C_\varphi$. So, if $C_\varphi$ is bounded, then $\varphi$ has a finite angular derivative $0<\lambda<\infty$ at infinity and
\begin{displaymath}
	\lambda \defeq \lim_{\stackrel{z \longrightarrow \infty}{\text{nontangentially}}} \frac{z}{\varphi(z)} = \lim_{\re(z) \rightarrow \infty} \frac{\re(z)}{\varphi(\re(z))} = a^{-1},
\end{displaymath}
where $0<a<\infty$ is defined as in \eqref{eq:nevanlinna}. Conversely, if $\varphi$ has a finite angular derivative $\lambda$ at infinity, then, by Theorem \ref{thm:elliott}, $C_\varphi$ is bounded on the Hardy space $H^2(\mathbb{C}_+)$, and, by Proposition \ref{prop:chalendar1}, we get that it is also bounded on $A^2_\nu$ with
\begin{displaymath}
	\|C_\varphi\|_{A^2_\nu} \leq \|C_{\psi_a}\|_{A^2_\nu} \|C_{\psi_{1/a}}\|_{H^2(\mathbb{C}_+)} \|C_\varphi\|_{H^2(\mathbb{C}_+)}.
\end{displaymath}
We can evaluate the RHS of this inequality using Theorem \ref{thm:elliott} and Proposition \ref{prop:chalendar2} to get
\begin{displaymath}
	\|C_\varphi\|_{A^2_\nu}^2 \leq \sup_{t>0} \frac{w(at)}{aw(t)} \cdot a  \cdot \lambda = \lambda \sup_{t>0} \frac{w(t/\lambda)}{w(t)}.
\end{displaymath}
By Corollary \ref{cor:nicecorollary} we also know that if $C_\varphi$ is bounded on $A^2_\nu$, then there exists $\alpha > 0$ such that
\begin{displaymath}
	\|C_\varphi\|_{A^2_\nu} \geq \|C_\varphi\|_{\mathcal{B}^2_\alpha(\mathbb{C}_+)} a^{(\alpha+1)/2} \left(\sup_{t>0}\frac{w(t/a)}{w(t)}\right)^{-1/2}.
\end{displaymath}
Again, we can evaluate the RHS of this inequality using Theorem \ref{thm:elliott} and Proposition \ref{prop:chalendar2} to get
\begin{displaymath}
	\|C_\varphi\|_{A^2_\nu}^2 \geq \lambda^{\alpha+2} a^{\alpha+1} \inf_{t>0} \frac{w(t)}{w(\lambda t)} = \lambda \inf_{t>0}\frac{w(t)}{w(\lambda t)}.
\end{displaymath}
\end{proof}

\section{Compactness}
\label{sec:compactness}

In \cite{elliott2011} S. J. Elliott and A. Wynn have proved that there exists no compact composition operator on any weighted Bergman space (Theorem 3.6). We shall extend this result to all Zen spaces. But first we need to prove two technical lemmata.

\begin{lemma}
Let $w$ be as given in \eqref{eq:w}. There exists $c\geq 2$ such that
\begin{equation}
	w\left(\frac{t}{2}\right) \leq cw(t) \qquad \qquad (\forall t>0).
\label{eq:wmestimate}
\end{equation}
\end{lemma}

\begin{proof}
This result follows from the \eqref{eq:delta2}-condition, and the proof uses essentially the same strategy as the proof of Theorem \ref{thm:awfultheorem}. We want to show that there exists $c\geq 2$ such that
\begin{displaymath}
	\int_0^\infty e^{-rt} \, d\tilde{\nu}(r) \leq c \int_0^\infty e^{-2rt} \, d\tilde{\nu}(r) \qquad \qquad (\forall t>0)
\end{displaymath}
or equivalently
\begin{displaymath}
	\int_0^\infty e^{-rt}(1-c e^{-rt}) \, d\tilde{\nu}(r) \leq 0 \qquad \qquad (\forall t>0).
\end{displaymath}
Again, consider the graph

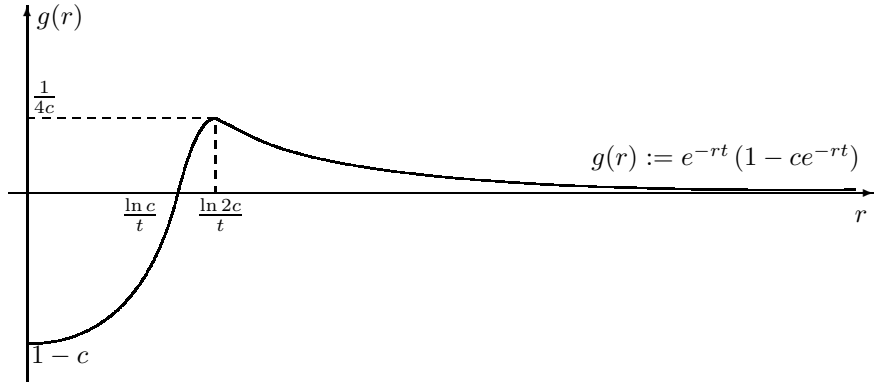
\setlength{\unitlength}{5mm}
\begin{picture}(25,12)(-0.5,-6)
\put(-0.5,0){\vector(1,0){23}}
\put(0,-5){\vector(0,1){10}}
\put(22,-0.75){$r$}
\put(0.25,4.5){$g(r)$}
\put(15,0.75){$g(r):=e^{-rt}\left(1-c e^{-rt}\right)$}
\put(0.1,-4.5){$1-c$}
\put(2.5,-0.75){$\frac{\ln c}{t}$}
\put(4.5,-0.75){$\frac{\ln 2c}{t}$}
\put(0.1,2.4){$\frac{1}{4c}$}
\multiput(0,2)(0.4,0){13}
{\line(1,0){0.2}}
\multiput(5,0)(0,0.4){5}
{\line(0,1){0.2}}
\linethickness{0.2mm}
\qbezier(0,-4)(3,-4)(4,0)
\qbezier(4,0)(4.5,2)(5,2)
\qbezier(5,2)(5,2)(6,1.5)
\qbezier(6,1.5)(9,0)(22,0.1)
\end{picture}
We need to have
\begin{displaymath}
	-\int_0^{\frac{\ln c}{t}} e^{-rt}\left(1-c e^{-rt}\right) \, d\tilde{\nu}(r) \geq \int_{\frac{\ln c}{t}}^\infty e^{-rt}\left(1-c e^{-rt}\right) \, d\tilde{\nu}(r).
\end{displaymath}
Observe that if $c \geq 2$, then we have 
\begin{equation}
	\frac{\ln 2c}{2t} \leq \frac{\ln c}{t} \qquad \qquad (\forall t>0).
\label{eq:logc}
\end{equation}
Let $R$ be given by \eqref{eq:delta2}. Then we have
\begin{align*}
	-\int_0^{\frac{\ln c}{t}} e^{-rt}\left(1-c e^{-rt}\right) \, d\tilde{\nu}(r) &\geq \tilde{\nu} \left[0, \frac{\ln 2c}{2t}\right) \left(c e^{-\ln 2c} - e^{-\frac{\ln 2c}{2}}\right) \\
	&\stackrel{\eqref{eq:delta2}}{\geq} \frac{\tilde{\nu}\left[0, \frac{\ln 2c}{t}\right)}{R} \left(\frac{1}{2} - \frac{1}{\sqrt{2c}}\right),
\end{align*}
and
\begin{align*}
	\int_{\frac{\ln c}{t}}^{\frac{\ln 2c}{t}} e^{-rt}\left(1-c e^{-rt}\right) \, d\tilde{\nu}(r) &\leq \tilde{\nu}\left[\frac{\ln c}{t}, \, \frac{\ln 2c}{t}\right) e^{-\ln 2c}\left(1-c e^{-\ln 2c}\right) \\
	& = \frac{1}{4c} \left(\tilde{\nu}\left[0, \, \frac{\ln 2c}{t}\right)-\tilde{\nu}\left[0, \, \frac{\ln c}{t}\right)\right) \\
	& \leq \frac{R-1}{4R c}\tilde{\nu}\left[0, \frac{\ln 2c}{t}\right),
	\end{align*}
because
\begin{displaymath}
	- \tilde{\nu}\left[0, \frac{\ln c}{t}\right) \stackrel{\eqref{eq:logc}}{\leq} - \tilde{\nu}\left[0, \frac{\ln 2c}{2t}\right) \stackrel{\eqref{eq:delta2}}{\leq} -\frac{\tilde{\nu}\left[0, \frac{\ln 2c}{t}\right)}{R},
\end{displaymath}
and for $c > R/2$ we have
\begin{align*}
\int_{\frac{\ln 2c}{t}}^\infty e^{-rt}\left(1-c e^{-rt}\right) \, d\tilde{\nu}(r) &\leq \sum_{k=0}^\infty \tilde{\nu}\left[2^k \frac{\ln 2c}{t}, \, 2^{k+1} \frac{\ln 2c}{t} \right) e^{-2^k\ln 2c}\left(1-c e^{-2^k\ln 2c}\right) \\
&\stackrel{\eqref{eq:delta2}}{\leq} (R-1) \tilde{\nu}\left[0, \, \frac{\ln 2c}{t} \right) \sum_{k=0}^\infty \frac{R^k}{(2c)^{2^k}}\left(1-\frac{c}{(2c)^{2^k}}\right) \\
&\leq (R-1) \tilde{\nu}\left[0, \, \frac{\ln 2c}{t} \right) \frac{1}{2c} \sum_{k=0}^\infty \left(\frac{R}{2c}\right)^k \\
&= \frac{R-1}{2c-R} \tilde{\nu}\left[0, \, \frac{\ln 2c}{t} \right).
\end{align*}
Putting these inequalities together, we get
\begin{displaymath}
	R \left(\frac{R-1}{2c-R} + \frac{R-1}{4R c}\right) \leq \frac{1}{2}-\frac{1}{\sqrt{2c}},
\end{displaymath}
which holds for sufficiently large $c$, since the LHS approaches 0 and the RHS approaches $1/2$ as $c$ goes to infinity.
\end{proof}

\begin{lemma}
The normalised reproducing kernels $k^{A^2_\nu}_z/\left\|k^{A^2_\nu}_z\right\|$ tend to 0 weakly as $z$ approaches infinity.
\end{lemma}

\begin{proof}
If $z$ approaches $\infty$ unrestrictedly, then either $\re(z) \longrightarrow \infty$ or $\re(z)<a$, for some $a>0$, and $\im(z) \longrightarrow \infty$. Suppose that $\re(z) \longrightarrow \infty$. Then
\begin{align*}
\lim_{\re(z) \rightarrow \infty} \left|k^{A^2_\nu}_z(\zeta)\right|/\left\|k^{A^2_\nu}_z\right\| &\stackrel{\eqref{eq:kernel}}{\leq} \lim_{\re(z) \rightarrow \infty} \int_0^\infty \frac{e^{-t(\re(z)+\re(\zeta))}}{w(t)} \, dt/\left\|k^{A^2_\nu}_z\right\| \\
&\leq \lim_{\re(z) \rightarrow \infty} \int_0^\infty \frac{e^{-t\re(z)}}{w(t)} \, dt/\left\|k^{A^2_\nu}_z\right\| \\
&= 2\lim_{\re(z) \rightarrow \infty} \int_0^\infty \frac{e^{-2t\re(z)}}{w(2t)} \, dt/\left\|k^{A^2_\nu}_z\right\| \\
&\stackrel{\eqref{eq:wmestimate}}{\lessapprox} \lim_{\re(z) \rightarrow \infty} \int_0^\infty \frac{e^{-2t\re(z)}}{w(t)} \, dt/\left\|k^{A^2_\nu}_z\right\| \\
&\stackrel{\eqref{eq:kernel}}{=} \lim_{\re(z) \rightarrow \infty} \left\|k^{A^2_\nu}_z\right\|^2/\left\|k^{A^2_\nu}_z\right\|  \\
&= \lim_{\re(z) \rightarrow \infty}\left\|k^{A^2_\nu}_z\right\| = 0.
\end{align*}
Otherwise $\re(z) \leq a$, for some $0 <a<\infty$, and
\begin{equation}
	\left\|k_z^{A^2_\nu}\right\|^2 \stackrel{\eqref{eq:kernel}}{=} \int_0^\infty \frac{e^{-2t\re(z)}}{w(t)} \, dt \geq \int_0^\infty \frac{e^{-2at}}{w(t)} \stackrel{\eqref{eq:kernel}}{=} \left\|k_a^{A^2_\nu}\right\|^2.
\label{eq:kera}
\end{equation}
Now,
\begin{displaymath}
	\int_0^\infty \left|\frac{e^{-t\overline{\zeta}}}{w(t)}\right| \, dt = \int_0^\infty \frac{e^{-2t\re(\zeta)/2}}{w(t)} \, dt \stackrel{\eqref{eq:kernel}}{=} \left\| k_{\frac{\zeta}{2}}^{A^2_\nu}\right\|^2 < \infty,
\end{displaymath}
so $e^{-\cdot \overline{\zeta}}/w(\cdot)$ is in $L^1(0, \, \infty)$, for all $\zeta \in \mathbb{C}_+$. Therefore, by the Riemann-Lebesgue Lemma for the Laplace transform (Theorem 1, p. 3 in \cite{bochner1949}), we get
\begin{equation}
\lim_{z \rightarrow \infty} \left|k^{A^2_\nu}_z(\zeta)\right| \stackrel{\eqref{eq:kernel}}{=} \lim_{z \rightarrow \infty} \left|\mathfrak{L}\left[\frac{e^{-\cdot \overline{\zeta}}}{w(\cdot)}\right](z)\right| = 0 \quad \Longrightarrow \quad \lim_{\stackrel{\im(z) \longrightarrow \infty}{0<\re(z)<a}} \left|k^{A^2_\nu}_z(\zeta)\right| = 0.
\label{eq:rl}
\end{equation}
And thus
\begin{displaymath}
	\lim_{\stackrel{\im(z)\longrightarrow \infty}{0<\re(z)<a}} \left|k^{A^2_\nu}_z(\zeta)\right|/\left\|k^{A^2_\nu}_z\right\| \stackrel{\eqref{eq:kera}}{\leq} \left\|k_a^{A^2_\nu}\right\|^{-1}\lim_{\stackrel{\im(z)\longrightarrow \infty}{0<\re(z)<a}} \left|k^{A^2_\nu}_z(\zeta)\right| \stackrel{\eqref{eq:rl}}{=} 0.
\end{displaymath}
So in either case we have
\begin{displaymath}
	\lim_{z\rightarrow \infty} \left|k^{A^2_\nu}_z(\zeta)\right|/\left\|k^{A^2_\nu}_z\right\| = 0.
\end{displaymath}
\end{proof}

\begin{definition}
The \emph{essential norm of an operator}, denoted $\| \cdot \|_e$ is the distance in the operator norm from the set of compact operators.
\end{definition}

\begin{theorem}
There is no compact composition operator on $A^2_\nu$.
\end{theorem}

\begin{proof}
Let $C_\varphi$ be a bounded operator on $A^2_\nu$, and denote its angular derivative at infinity by $\lambda$. For any $\delta>0$ we can choose a compact operator $Q$ such that $\|C_\varphi \|_e + \delta \geq \|C_\varphi-Q\|$. By the previous lemma, the sequence $k^{A^2_\nu}_z/\|k^{A^2_\nu}_z\|$ tends to 0 weakly, as $z$ approaches infinity, so $Q^*\left(k^{A^2_\nu}_z/\|k^{A^2_\nu}_z\|\right) \longrightarrow 0$, and consequently
\begin{align*}
\|C_\varphi \|_e + \delta &\geq \|C_\varphi-Q\| \geq \limsup_{z \longrightarrow \infty} \frac{\left\|(C_\varphi-Q)^*k^{A^2_\nu}_z\right\|}{\left\|k^{A^2_\nu}_z\right\|} \\
&= \limsup_{z \longrightarrow \infty} \frac{\left\|C^*_\varphi k^{A^2_\nu}_z\right\|}{\left\|k^{A^2_\nu}_z\right\|} = \limsup_{z \longrightarrow \infty} \frac{\left\| k^{A^2_\nu}_{\varphi(z)}\right\|}{\left\|k^{A^2_\nu}_z\right\|}.
\end{align*}
Suppose, for contradiction, that $C_\varphi$ is compact, then the last quantity above must be equal to 0, and hence the limit of $\| k^{A^2_\nu}_{\varphi(z)}\|/\|k^{A^2_\nu}_z\|$ exists and is also equal to 0. That is, for each $\varepsilon>0$ there exists $z_0 \in \mathbb{C}_+$ such that
\begin{equation}
	\frac{\left\| k^{A^2_\nu}_{\varphi(z)}\right\|}{\left\|k^{A^2_\nu}_z\right\|} \stackrel{\eqref{eq:kernel}}{=} \sqrt{\frac{\int_0^\infty \frac{e^{-2t\re(\varphi(z))}}{w(t)} dt}{\int_0^\infty \frac{e^{-2t\re(z)}}{w(t)} dt}} < \varepsilon,
\label{eq:epsilon}
\end{equation}
for all $z \in \mathbb{C}_+$ with $|z| \geq |z_0|$. Since $C_\varphi$ is bounded and
\begin{displaymath}
	\lambda = \limsup_{z\longrightarrow\infty} \frac{\re(z)}{\re(\varphi(z))},
\end{displaymath}
for any $0<\kappa<\lambda$ there exists a sequence $(z_j)_{j=1}^\infty$ with $|z_j| \geq |z_0|$, for all $j \geq 0$, such that
\begin{equation}
	\frac{\re(z)}{\re(\varphi(z))} > \kappa \qquad \qquad \left(\forall z \in \{z_j\}_{j=1}^\infty \right).
\label{eq:kappa}
\end{equation}
Let $\psi(z)=\kappa z$. If $z \in \left\{z_j\right\}_{j=1}^\infty$, then
\begin{displaymath}
	\|C_\psi\|^2 \geq \frac{\left\|k^{A^2_\nu}_{\psi(\varphi(z))}\right\|^2}{\left\|k^{A^2_\nu}_{\varphi(z)}\right\|^2} \stackrel{\eqref{eq:kernel}}{=} \frac{\int_0^\infty \frac{e^{-2t\kappa\re(\varphi(z))}}{w(t)} dt}{\int_0^\infty \frac{e^{-2t\re(\varphi(z))}}{w(t)} dt} \stackrel{\eqref{eq:kappa}}{\geq} \frac{\int_0^\infty \frac{e^{-2\re(z)}}{w(t)} dt}{\int_0^\infty \frac{e^{-2t\re(\varphi(z))}}{w(t)} dt} \stackrel{\eqref{eq:epsilon}}{>} \frac{1}{\varepsilon^2},
\end{displaymath}
which is absurd, because $C_\psi$ is bounded on $A^2_\nu$ by Theorem \ref{thm:Zenboundedness} and $\varepsilon$ in the expression above can be chosen to be arbitrarily large. So $\|C_\varphi\|_e > 0$, and consequently $C_\varphi$ is not compact.
\end{proof}

\textbf{Acknowledgement} The author of this article would like to thank the UK Engineering and Physical Sciences Research Council (EPSRC) and the School of Mathematics at the University of Leeds for their financial support. He is also very grateful to Professor Jonathan R. Partington for all the important comments and help in preparation of this research paper.

\end{document}